\documentclass[a4paper]{amsart}
\usepackage{amssymb, enumitem}
\usepackage[all]{xy}
\usepackage{hyperref, aliascnt}

\newtheorem{lma}{Lemma}[section]

\newaliascnt{thmCt}{lma}
\newtheorem{thm}[thmCt]{Theorem}
\aliascntresetthe{thmCt}

\newaliascnt{corCt}{lma}
\newtheorem{cor}[corCt]{Corollary}
\aliascntresetthe{corCt}

\newaliascnt{prpCt}{lma}
\newtheorem{prp}[prpCt]{Proposition}
\aliascntresetthe{prpCt}

\theoremstyle{definition}

\newaliascnt{dfnCt}{lma}
\newtheorem{dfn}[dfnCt]{Definition}
\aliascntresetthe{dfnCt}

\newaliascnt{rmkCt}{lma}
\newtheorem{rmk}[rmkCt]{Remark}
\aliascntresetthe{rmkCt}

\newaliascnt{pbmCt}{lma}
\newtheorem{pbm}[pbmCt]{Problem}
\aliascntresetthe{pbmCt}

\newaliascnt{ntnCt}{lma}
\newtheorem{ntn}[ntnCt]{Notation}
\aliascntresetthe{ntnCt}

\newcounter{theoremintro}

\newaliascnt{thmIntroCt}{theoremintro}
\newtheorem{thmIntro}[thmIntroCt]{Theorem}
\aliascntresetthe{thmIntroCt}

\newaliascnt{dfnIntroCt}{theoremintro}
\newtheorem{dfnIntro}[dfnIntroCt]{Definition}
\aliascntresetthe{dfnIntroCt}

\newaliascnt{corIntroCt}{theoremintro}
\newtheorem{corIntro}[corIntroCt]{Corollary}
\aliascntresetthe{corIntroCt}

\def\today{\number\day\space\ifcase\month\or   January\or February\or
   March\or April\or May\or June\or   July\or August\or September\or
   October\or November\or December\fi\   \number\year}

\newcommand{\andSep}{\,\,\,\text{ and }\,\,\,}
\newcommand{\Bdd}{{\mathcal{B}}}
\newcommand{\ca}{$C^*$-algebra}

\title{Weighted homomorphisms between C*-algebras}
\date{\today}

\author{Eusebio Gardella}
\address{Eusebio Gardella,
Department of Mathematical Sciences,
Chalmers University of Technology and University of Gothenburg,
Gothenburg SE-412 96, Sweden.}
\email{gardella@chalmers.se}
\urladdr{www.math.chalmers.se/~gardella}

\author{Hannes Thiel}
\address{Hannes Thiel,
Department of Mathematics,
Kiel University,
24118 Kiel,
Germany}
\email{hannes.thiel@math.uni-kiel.de}
\urladdr{www.hannesthiel.org}

\thanks{
The first named author was partially supported by the Deutsche Forschungsgemeinschaft (DFG, German Research Foundation) through an Eigene Stelle, and by the Swedish Research Council Grant 2021-04561.
The second named author was partially supported by the ERC Consolidator Grant No.~681207.
}

\subjclass[2010]%
{Primary
47A65, 
47B49; 
Secondary
46L05, 
47A20, 
47B65.  
}
\date{\today}

\begin{document}

\begin{abstract}
We show that a bounded, linear map between \ca{s} is a weighted $\ast$-homomorphism (the central compression of a $\ast$-homomorphism) if and only if it preserves zero-products, range-orthogonality, and domain-orthogonality.
It follows that a self-adjoint, bounded, linear map is a weighted $\ast$-homomorphism if and only if it preserves zero-products.

As an application we show that a linear map between \ca{s} is completely positive, order zero in the sense of Winter-Zacharias if and only if it is positive and preserves zero-products.
\end{abstract}

\maketitle

\section{Introduction}

The study of bounded, linear maps between \ca{s} that preserve various types of orthogonality has a long history.
One of the earliest results, due to Arendt \cite{Are83SpectralLampertiOps}, 
is the description of such maps between unital, commutative \ca{s}: 
every bounded, linear map $\varphi\colon C(X)\to C(Y)$ that preserves zero-products (such maps are also called \emph{disjointness preserving}, or \emph{separating}) is spatially implemented, that is, setting $h:=\varphi(1)\in C(Y)$ and $U:=\{y\in Y: h(y)\neq 0\}$, there exists a continuous map $p\colon U\to X$ such that $\varphi(f)(y)=h(y)f(p(y))$ for every $y\in Y$.
Note that $p$ induces a $\ast$-homomorphism $\pi\colon C(X)\to C_b(U)$, $f\mapsto f\circ p$.
By embedding $C(Y)$ and $C_b(U)$ into a suitably larger (commutative) \ca{} -- the canonical choice is $C(Y)^{**}$ -- we obtain that $\varphi(f)=h\pi(f)=\pi(f)h$ for all $f\in C(X)$.
In particular, $\varphi$ is the compression of the $\ast$-homomorphism $\pi$ by $h$.
For this description of $\varphi$ as the compression of a $\ast$-homomorphism, it is crucial to allow for an enlarged target algebra.
Indeed, an example of Wolff \cite{Wol94DisjPres}, shows that the image of $\pi$ is not in
general contained in $C(Y)$.  

Generalizing this idea to the noncommutative setting, we define:

\begin{dfnIntro}
\label{dfn:WeightedStarHomo}
We say that a map $\varphi\colon A\to B$ between \ca{s} is a \emph{weighted $\ast$-homomorphism} if there exist a \ca{} $D$ containg $B$, a $\ast$-homomorphism $\pi\colon A\to D$ and an element $h\in D$ such that $\varphi(a)=\pi(a)h=h\pi(a)$ for all $a\in A$.
\end{dfnIntro}

By the above-mentioned result of Arendt, it follows that a bounded, linear map between unital, commutative \ca{s} is a weighted $\ast$-homomorphism if and only if it preserves zero-products.
The main result of this paper is a generalization of this characterization to the noncommutative setting; see \autoref{prp:CharWeighted}.
In this case, we also need to consider the following types of orthogonality:
two elements $a$ and $b$ in a \ca{} are said to be \emph{range-orthogonal} if $a^*b=0$;
they are \emph{domain-orthogonal} if $ab^*=0$.

\begin{thmIntro}[{\ref{prp:CharWeighted}}]
\label{thmB}
A bounded, linear map between \ca{s} is a weighted \mbox{$\ast$-ho\-mo}\-mor\-phism if and only if it preserves range-orthogonality, domain-orthog\-o\-nality and zero-products.
\end{thmIntro}

Note that two elements in a \emph{commutative} \ca{} are range-orthogonal if and only if they are domain-orthogonal, and if and only if they have zero-product.
In particular, our \autoref{thmB} generalizes the description of bounded, linear mpas between unital, commutative \ca{s} that preserve zero-products.

Bounded, linear maps between (noncommutative) \ca{s} that preserve zero-products have been extensively studied by Wong and coauthors, \cite{CheKeLeeWon03PresZeroProd, KeLiWon04ZeroProdPresOpValuedFcts, Won07ZeroPordPres}, and a general description of the structure of such maps seems out of reach at the moment.
However, a zero-product preserving linear map which is additionally self-adjoint (equivalently, $\ast$-preserving) automatically preserves domain and range orthogonality, and hence we obtain:

\begin{corIntro}[\ref{prp:CharWeightedSelfadjoint}]
\label{CorC}
A self-adjoint, bounded, linear map between \ca{s} is a weighted $\ast$-homomorphism if and only if it preserves zero-products.
\end{corIntro}

The above corollary is the analog of a result of Wolff, \cite{Wol94DisjPres}, who showed that a self-adjoint, bounded linear map between \ca{s} that preserves zero-products \emph{among self-adjoint elements} is a weigthed Jordan $\ast$-homomorphisms.
(Wolff showed the result under the assumption that the domain is unital.
This condition can be removed using \cite[Lemma~2.2]{Won07ZeroPordPres}.)

In \cite{WinZac09CpOrd0}, Winter and Zacharias introduced completely positive maps of \emph{order zero} as those completely positive  maps between \ca{s} that preserve zero-products among positive elements (equivalently, among self-adjoint elements). 
The main result of \cite{WinZac09CpOrd0} is a characterization of completely positive maps of order zero as the \emph{positively} weighted $\ast$-homomorphisms.
These maps play an important role in the fine structure theory of nuclear \ca{s};
see \cite{HirKirWhi12DecApproxNucl, WinZac10NuclDim}.

Using that a positive, linear map between \ca{s} is automatically bounded and self-adjoint, we obtain the following new characterization of completely positive, order zero maps:

\begin{corIntro}[\ref{prp:CharOrderZero}]
A map between \ca{s} is completely positive, order zero if and only if it is positive and preserves zero-products.

In particular, every positive, zero-product preserving map between \ca{s} is completely positive (and order zero).
\end{corIntro}

In \autoref{sec:preservingOrth}, we study bounded, linear maps between \ca{s} that preserve zero-products, range-orthogonality, or domain-orthogonality and we provide characterizations of these properties in terms of algebraic identities.
For example, a bounded, linear map $\varphi$ preserves zero-products if and only if $\varphi(ab)\varphi(c)=\varphi(a)\varphi(bc)$ for all $a,b,c$ in the domain;
see \autoref{prp:CharOrthPres}.
Using this, we show that preserving these types of orthogonality passes to bitransposes and tensor products.

We also consider maps preserving zero-TRO-products or `usual' orthogonality.

In \autoref{sec:support}, we consider a bounded, linear map $\varphi\colon A\to B$ between \ca{s}, and a $\ast$-representation $B\subseteq\Bdd(H)$ of $B$ on some Hilbert space $H$.
Assuming that~$\pi$ preserves zero-products, range-orthogonality, and domain-orthogonality, we construct a (canonical) $\ast$-homomorphism $\pi_\varphi\colon A\to\Bdd(H)$ that takes image in the bicommutant $B''\subseteq\Bdd(H)$ and that satisfies $\pi_\varphi(a)\varphi(b)=\varphi(ab)$ and $\varphi(a)\pi_\varphi(b)=\varphi(ab)$ for all $a,b\in A$.
This is the main technical step in the proof of our main result, which we complete in \autoref{sec:weightedHomo}.

\section{Maps preserving different notions of orthogonality}
\label{sec:preservingOrth}

In this section, we study bounded linear maps between \ca{s} that preserve certain notions of orthogonality. 
We begin by defining these notions.

\begin{dfn}
Let $A$ be a \ca, and let $a,b,c\in A$. We say that
\begin{itemize} 
\item
$a$ and $b$ have \emph{zero-product} if $ab=0$,
\item $a$ and $b$ are \emph{range-orthogonal} if $a^*b=0$;
\item
$a$ and $b$ are \emph{domain-orthogonal} if $ab^*=0$;
\item
$a$ and $b$ are \emph{orthogonal} if $a^*b=ab^*=0$;
\item $a, b$ and $c$ have \emph{zero-TRO-product} if $ab^*c=0$.
\end{itemize} 
\end{dfn}

We will be interested in maps preserving certain combinations of the above notions. 
More explicitly, we say that a map $\varphi\colon A\to B$ between \ca{s} preserves:
\begin{enumerate}
\item
zero-products if $ab=0$ implies $\varphi(a)\varphi(b)=0$
\item
range-orthogonality if $a^*b=0$ implies $\varphi(a)^*\varphi(b)=0$;
\item
domain-orthogonality if $ab^*=0$ implies $\varphi(a)\varphi(b)^*=0$;
\item
orthogonality if $ab^*=a^*b=0$ implies $\varphi(a)\varphi(b)^*=\varphi(a)^*\varphi(b)=0$;
\item
zero-TRO-products if $ab^*c=0$ implies $\varphi(a)\varphi(b)^*\varphi(c)=0$.
\end{enumerate}

We provide characterizations of (1), (2), (3) and (5) in terms of algebraic equations (\autoref{prp:CharOrthPres}), which we then use to show that each of those properties passes to the bitranspose (\autoref{prp:Bitranspose}) and to tensor products (\autoref{prp:TensProd}).
In particular, if a bounded linear map $A\to B$ preserves zero-products, then so does every amplification $M_n(A)\to M_n(B)$ -- and analogously for maps preserving range-orthogonality, preserving domain-orthogonality, or preserving zero-TRO-products;
see \autoref{prp:CompletelyOrthPres}.
In particular, a zero-product preserving map is automatically completely zero-product preserving (and similarly for range-orthogonality, domain-orthogonality, and zero-TRO-products).

The situation for orthogonality-preserving maps is different:
As noted in \cite[Section~4]{Gar20CompleteOrthPres}, the transpose map $\tau$ on $M_2(\mathbb{C})$ preserves orthogonality, but its amplification $\tau^{(2)}$ does not.
Thus, orthogonality-preservation does not pass to tensor products.
Nevertheless, using the structure result for orthogonality-preserving maps obtained 
in \cite{BurFerGarMarPer08OrthPresJB}, we show that the bitranspose of an orthogonality-preserving map is again orthogonality-preserving;
see \autoref{prp:Bitranspose}.

Part~(1) of the next result follows from \cite[Lemma~3.4]{AlaBreExtVil09MapsPresZP}.
(Note that by Examples~1.3(2) and Theorem~2.11 in \cite{AlaBreExtVil09MapsPresZP}, every \ca{} has property~$(\mathbb{B})$ introduced in \cite{AlaBreExtVil09MapsPresZP}.)
The proof in \cite[Lemma~3.4]{AlaBreExtVil09MapsPresZP} is based on the fact that every element in a unital \ca{} is a linear combination of unitaries and that unitaries are doubly power-bounded operators.
For maps between von Neumann algebras (and more generally, \ca{s} of real rank zero), Part~(1) of the next result has also been obtained in \cite[Theorem~4.1]{CheKeLeeWon03PresZeroProd}, using that the linear span of projections is dense.

Our proof is inspired by \cite[Lemma~4.4]{CheKeLeeWon03PresZeroProd}, and uses methods that go back to Wolff, \cite{Wol94DisjPres}.
These methods allow us to also characterize maps that preserve range-orthogonality or domain-orthogonality.

\begin{thm}
\label{prp:CharOrthPres}
Let $\varphi\colon A\to B$ be a bounded, linear map between \ca{s}.
Then:
\begin{enumerate}
\item
$\varphi$ preserves zero-products if and only if $\varphi(ab)\varphi(c)=\varphi(a)\varphi(bc)$ for all $a,b,c\in A$.
\item
$\varphi$ preserves range-orthogonality if and only if $\varphi(b^*a)^*\varphi(c)=\varphi(a)^*\varphi(bc)$ for all $a,b,c\in A$.
\item
$\varphi$ preserves domain-orthogonality if and only if $\varphi(ab)\varphi(c)^*=\varphi(a)\varphi(cb^*)^*$ for all $a,b,c\in A$.
\item
$\varphi$ preserves zero-TRO-products if and only if the equality $\varphi(ab)\varphi(c)^*\varphi(de)=\varphi(a)\varphi(d^*cb^*)^*\varphi(e)$ holds for all $a,b,c,d,e\in A$.
\end{enumerate} 
\end{thm}
\begin{proof}
We only prove statement~(2).
Statements~(1), (3) and~(4) are shown analogously.

\emph{Backward implication:}
Let $a,b\in A$ satisfy $a^*b=0$.
We need to verify that $\varphi(a)^*\varphi(b)=0$.
Let $\varepsilon>0$.
Using an approximate identity in $A$, we can choose $c\in A$ satisfying $\|b-bc\|<\varepsilon$.
Using that $b^*a=0$ and applying the assumption at the first step, we get
\[
\varphi(a)^*\varphi(bc)
= \varphi(b^*a)^*\varphi(c)
=0
\]
and therefore
\[
\|\varphi(a)^*\varphi(b)\|
= \|\varphi(a)^*\varphi(b) - \varphi(a)^*\varphi(bc)\|
\leq \|\varphi(a)^*\| \|\varphi\| \|b-bc\|
\leq \varepsilon \|\varphi(a)^*\| \|\varphi\|.
\]
Since $\varepsilon>0$ was arbitrary, we get $\varphi(a)^*\varphi(b)=0$, as desired.

\emph{Forward implication:}
We assume that $\varphi$ preserves range-orthogonality.
Given $a,b,c\in A$, we need to show that 
\[
\varphi(b^*a)^*\varphi(c)=\varphi(a)^*\varphi(bc).
\]
Using linearity of $\varphi$, we may assume that $b$ is self-adjoint and satisfies $\|b\|<1$.

For $k\geq 2$, let $f_k\colon\mathbb{R}\to[0,1]$ be the continuous function that takes the value $0$ on $(-\infty,\tfrac{1}{k}]\cup[1+\tfrac{1}{k},\infty)$, that takes the value $1$ on $[\tfrac{2}{k},1]$, and that is affine on $[\tfrac{1}{k},\tfrac{2}{k}]$ and on $[1,1+\tfrac{1}{k}]$.
Let $m\geq 1$.
For each $j\in\{-m,\ldots,m-1\}$ and $k\geq 2$, we define $f_{j,k}\colon\mathbb{R}\to[0,1]$ by
\[
f_{j,k}(t) := f_k(mt-j)
\]
for $t\in\mathbb{R}$. Note that the support of $f_{j,k}$ is 
$\big(\frac{j}{m}+\frac{1}{mk}, \tfrac{j+1}{m}+\frac{1}{mk}\big)$. 
In particular, $f_{j,k}$ and $f_{j',k'}$ are orthogonal if 
$k\neq k'$.
Moreover, the sequence $(f_{j,k})_k$ converges pointwise to the characteristic function of $\big(\tfrac{j}{m},\tfrac{j+1}{m}\big]$.

Let $\Phi:=\varphi^{**}\colon A^{**}\to B^{**}$ denote the bitranspose of $\varphi$.
In $A^{**}$, the sequence $(f_{j,k}(b))_k$ converges weak* to a projection $e_j$.
\vspace{.2cm}

\textbf{Claim~1:} \emph{Let $j>j'$.
Then $\Phi(e_j^*a)^*\Phi(e_{j'}c)=0$.}
To prove the claim, let $k\geq 2$.
Given $k'$ with $k\leq k'$, we have $(f_{j,k}(b)^*a)^*(f_{j',k'}(b)c)=0$
since $f_{j,k}f_{j',k'}=0$, and thus
\[
\varphi(f_{j,k}(b)^*a)^* \varphi(f_{j',k'}(b)c)=0.
\]
Using that $\Phi$ is weak*-continuous, and that multiplication in $B^{**}$ is separately weak*-continuous, we get
\[
\varphi(f_{j,k}(b)^*a)^*\Phi(e_{j'}c)
= \text{wk*-}\lim_{k'} \varphi(f_{j,k}(b)^*a)^*\varphi(f_{j',k'}(b)c)
= 0.
\]
Since this holds for every $k\geq 2$, we get
\[
\Phi(e_j^*a)^*\Phi(e_{j'}c)
= \text{wk*-}\lim_{k} \Phi(f_{j,k}(b)^*a)^*\varphi(e_{j'}c)
= 0.
\]
This proves the claim.
\vspace{.2cm}

\textbf{Claim~2:} \emph{Let $j<j'$.
Then $\Phi(e_j^*a)^*\Phi(e_{j'}c)=0$.}
The proof is analogous to that of the previous claim.
For fixed $k'\geq 2$, and $k\geq k'$, we have $(f_{j,k}(b)^*a)^*(f_{j',k'}(b)c)=0$, and thus $\varphi(f_{j,k}(b)^*a)^* \varphi(f_{j',k'}(b)c)=0$.
We then get
\[
\Phi(e_j^*a)^*\varphi(f_{j',k'}(b)c)
= \text{wk*-}\lim_{k'} \varphi(f_{j,k}(b)^*a)^*\varphi(f_{j',k'}(b)c)
= 0.
\]
Since this holds for every $k'\geq 2$, we get
\[
\Phi(e_j^*a)^*\Phi(e_{j'}c)
= \text{wk*-}\lim_{k} \varphi(f_{j,k}(b)^*a)^*\Phi(e_{j'}c)
= 0.
\]
This proves the claim.
\vspace{.2cm}

Using that $b$ is selfadjoint and $\|b\|<1$, we have
\[
\Big\| b - \sum_{j=-m}^{m-1} \tfrac{j}{m} e_j \Big\| \leq \tfrac{1}{m}.
\]

Moreover, we have $\sum_j e_j = 1$ and therefore $a=\sum_j e_j^*a$ and $c=\sum_j e_jc$.
Set $K:=\|\varphi\|^2\|a\|\|c\|$.
We write $x\approx_\varepsilon y$ to mean $\|x-y\|\leq\varepsilon$.
Using Claims~1 and~2 at the second and third steps, we get
\begin{align*}
\varphi(b^*a)^*\varphi(c)
&\approx_{\frac{K}{m}} \Phi\Big( \sum_{j=-m}^{m-1} \tfrac{j}{m} e_j^*a \Big)^*
\Phi\Big( \sum_{j=-m}^{m-1} e_j c \Big)\\
&= \sum_{j=-m}^{m-1} \tfrac{j}{m} \Phi(e_j^*a)^*\Phi( e_j c) \\
&= \Phi\Big( \sum_{j=-m}^{m-1} e_j^*a \Big)^*
\Phi\Big( \sum_{j=-m}^{m-1} \tfrac{j}{m} e_j c \Big)\\
&\approx_{\frac{K}{m}} \varphi(a)^*\varphi(bc).
\end{align*}
Since this holds for every $m\geq 1$, we get the desired equality.
\end{proof}

Peralta showed in \cite[Proposition~3.7]{Per15Note2Local} that if $\varphi\colon A\to B$ preserves zero-products, then so does the restriction of the bitranspose $\varphi^{**}\colon A^{**}\to B^{**}$ to the multiplier algebra $M(A)$.
We generalize this by showing that $\varphi^{**}$ preserves in fact all zero-products in $A^{**}$.
We obtain similar results for range-orthogonality, domain-orthogonality and zero-TRO-products.
Using the structure result for orthogonality-preserving maps from \cite{BurFerGarMarPer08OrthPresJB}, we show that orthogonality-preservation also passes to bitransposes.

\begin{prp}
\label{prp:Bitranspose}
Let $\varphi\colon A\to B$ be a bounded, linear map between \ca{s}.
If~$\varphi$ preserves zero-products (range-orthogonality, domain-orthogonality, zero-TRO-products, orthogonality), then so does the bitranspose $\varphi^{**}\colon A^{**}\to B^{**}$.
\end{prp}
\begin{proof}
Set $\Phi=\varphi^{**}\colon A^{**}\to B^{**}$.
We first show the result for the case that $\varphi$ preserves zero-products.
By \autoref{prp:CharOrthPres}(1), we have $\varphi(ab)\varphi(c)=\varphi(a)\varphi(bc)$ for every $a,b,c\in A$.
For $a,b,c\in A^{**}$, we denote by $E(a,b,c)$ the (potentially false) identity
\begin{align}\label{eqn:Phi}\tag*{$E(a,b,c)$}
 \Phi(ab)\Phi(c)=\Phi(a)\Phi(bc).
\end{align}
By assumption, $E(a,b,c)$ holds whenever $a,b,c\in A$.
In three steps, we will show that it holds for all $a,b,c\in A^{**}$.

Let $a\in A^{**}$ and $b,c\in A$ be given. 
Choose a net $(a_\lambda)_\lambda$ in $A$ that converges weak* to $a$, and note 
that $E(a_\lambda,b,c)$ is true.
Using at the first step that $\Phi$ is weak*-continuous and 
that the multiplication operations in $A^{**}$ and $B^{**}$ are separately weak*-continuous, 
and using $E(a_\lambda,b,c)$ at the second step, we get
\[
\Phi(ab)\Phi(c)
= \text{wk*-}\lim_\lambda \Phi(a_\lambda b)\Phi(c)
=\text{wk*-}\lim_\lambda \Phi(a_\lambda)\Phi(bc)
= \Phi(a)\Phi(bc),
\]
so $E(a,b,c)$ holds whenever $a\in A^{**}$ and $b,c\in A$.

Given $a,b\in A^{**}$ and $c\in A$, we choose as above a net $(b_\lambda)_\lambda$ in 
$A$ converging weak* to $b$. By the previous paragraph, $E(a,b_\lambda,c)$ is true
for all $\lambda$. Using weak*-continuity of $\Phi$ and the fact 
that the multiplication operations in $A^{**}$ and $B^{**}$ are separately weak*-continuous,
we deduce that \eqref{eqn:Phi} holds for the triple $(a,b,c)$.
The general case for arbitrary $a,b,c\in A^{**}$ is proved similarly. 
We deduce that $\Phi$ preserves zero-products.

Using the characterizations of maps preserving range-orthogonality, domain-orthogonality or zero-TRO-products from \autoref{prp:CharOrthPres}, an analogous argument shows that 
$\Phi$ preserves range-orthogonality, domain-orthogonality, or zero-TRO-prod\-ucts, whenever $\varphi$ does.

Lastly, assume that $\varphi$ preserves orthogonality.
Recall that the triple product in a \ca{} is defined as $\{a,b,c\}=\tfrac{1}{2}(ab^*c+cb^*a)$, and that a bounded linear map~$\theta$ between \ca{s} is a \emph{triple homomorphism} if $\theta(\{a,b,c\})=\{\theta(a),\theta(b),\theta(c)\}$ for all $a,b,c$ in the domain of $\theta$.
Equivalently, $\theta$ satisfies
\[
\theta(ab^*c+cb^*a)
= \theta(a)\theta(b)^*\theta(c)+\theta(c)\theta(b)^*\theta(a)
\]
for all $a,b,c$.
Using an argument similar to the one above, it follows that the bitranspose $\theta^{**}$ is a triple homomorphism as well.

By \cite[Lemma~1]{BurFerGarMarPer08OrthPresJB}, two elements $a,b$ in a \ca{} are orthogonal if and only if $\{a,a,b\}=0$.
It follows that every triple homomorphism preserves orthogonality.

Set $h:=\Phi(1)\in B^{**}$.
Let $h=v|h|$ be the polar decomposition of $h$.
The partial isometry $v$ is also called the \emph{range tripotent} of $h$, and it is 
denoted by $r(h)$ in \cite{BurFerGarMarPer08OrthPresJB}.
By \cite[Theorem~17]{BurFerGarMarPer08OrthPresJB}, there exists a triple homomorphism $\pi\colon A\to B^{**}$ such that
\[
\varphi(a)
= hv^*\pi(a)
= \pi(a)v^*h
\]
for all $a\in A$.
Set $\Pi:=\pi^{**}\colon A^{**}\to B^{**}$, which is readily checked to be a triple 
homomorphism
It follows that
\[
\Phi(a)
= hv^*\Pi(a)
= \Pi(a)v^*h
\]
for all $a\in A^{**}$.

To show that $\Phi$ preserves orthogonality, let $a,b\in A^{**}$ be orthogonal, that is, $ab^*=a^*b=0$.
Since $\Pi$ is a triple homomorphism and therefore preserves orthogonality, we get $\Pi(a)\Pi(b)^*=\Pi(a)^*\Pi(b)=0$.
Using this at the last steps, we get
\[
\Phi(a)\Phi(b)^*
= (hv^*\Pi(a))(hv^*\Pi(b))^*
= hv^*\Pi(a)\Pi(b)^*vh^*
= 0
\]
and
\[
\Phi(a)^*\Phi(b)
= (\Pi(a)v^*h)^*(\Pi(b)v^*h)
= h^*v\Pi(a)^*\Pi(b)v^*h
= 0,
\]
as desired.
\end{proof}

In the next result, we use $\otimes$ to denote the minimal tensor product of \ca{s}.
We note that \autoref{prp:TensProd} also holds for maximal tensor products, with essentially the same proof.

\begin{prp}
\label{prp:TensProd}
Let $\varphi\colon A\to B$ and $\psi\colon C\to D$ be bounded, linear maps between \ca{s}.
If $\varphi$ and $\psi$ preserve zero-products (range-orthogonality, domain-orthogonality, zero-TRO-products), then so does the tensor product map $\varphi\otimes\psi\colon A\otimes C\to B\otimes D$.
\end{prp}
\begin{proof}
We prove the results for zero-product preserving maps.
The statements for range-orthogonality and domain-orthogonality preserving maps are shown analogously.

Assume that $\varphi$ and $\psi$ preserve zero-products and set $\alpha:=\varphi\otimes\psi\colon A\odot B\to C\odot D$.
Using that $\varphi$ and $\psi$ satisfy the formula from \autoref{prp:CharOrthPres}(1), we verify that $\alpha$ satisfies the formula as well.
Let 
\[
a=\sum_j u_j\otimes v_j, \quad
b=\sum_k w_k\otimes x_k, \andSep
c=\sum_l y_l\otimes z_l
\]
be finite sums of simple tensors in $A\otimes B$.
Then
\begin{align*}
\alpha(ab)\alpha(c)
&=\Big( \varphi\otimes\psi\Big( \sum_{j,k} u_jw_k\otimes v_jx_k \Big)\Big)\Big(\varphi\otimes\psi\Big( \sum_l y_l\otimes z_l \Big)\Big) \\
&= \sum_{j,k,l}\varphi(u_jw_k)\varphi(y_l)\otimes\psi(v_jx_k)\psi(z_l) \\
&= \sum_{j,k,l}\varphi(u_j)\varphi(w_ky_l)\otimes\psi(v_j)\psi(x_kz_l) 
= \alpha(a)\alpha(bc).
\end{align*}

Using that finite sums of simple tensors are dense in $A\otimes B$, and that $\alpha$ is continuous, it follows that $\alpha(ab)\alpha(c)=\alpha(a)\alpha(bc)$ for all $a,b,c\in A\otimes B$.
Applying \autoref{prp:CharOrthPres}(1), it follows that $\alpha$ preserves zero-products.
\end{proof}

The next result shows that every zero-product preserving map is automatically `completely zero-product preserving' (and similarly for range-orthogonality preserving, domain-orthogonality preserving and zero-TRO-product preserving).
For orthogonality-preserving maps, this is not the case:
the transpose map on $M_2(\mathbb{C})$ is orthogonality-preserving, but its amplifications are not.

\begin{cor}
\label{prp:CompletelyOrthPres}
Let $\varphi\colon A\to B$ be a bounded, linear map between \ca{s}.
If~$\varphi$ preserves zero-products (range-orthogonality, domain-orthogonality, zero-TRO-product), then so does the amplification $\varphi^{(n)}\colon M_n(A)\to M_n(B)$ for every $n\geq 1$.
\end{cor}
\begin{proof}
This follows by taking $C=D=M_n$ and $\psi=\mathrm{id}_{M_n}$ in \autoref{prp:TensProd}.
\end{proof}

Some of the statements in this section do not involve the adjoint operation in a \ca, and therefore make sense in more general settings. 
It would thus be interesting to find other classes of Banach algebras for which part~(1) of~\autoref{prp:CharOrthPres} or \autoref{prp:CompletelyOrthPres} are true. 
For example, one could explore these questions in the context of $L^p$-operator algebras \cite{Gar21ModernLp}, or at least for the well-behaved class arising from groups as in \cite{GarThi15GpAlgLp, GarThi19ReprConvLq}.

\section{Constructing the support \texorpdfstring{$\ast$}{*}-homomorphisms}
\label{sec:support} 

Throughout this section, we assume that $\varphi\colon A\to B$ is a bounded, linear map between \ca{s} that preserves zero-products, range-orthogonality, and also domain-orthogonality. By 
\autoref{prp:CharOrthPres}, for all $a,b,c\in A$ we have
\begin{align}
\label{eq:ZP}\tag{3.1}
& \varphi(ab)\varphi(c)=\varphi(a)\varphi(bc); \\
\label{eq:RO}\tag{3.2}
& \varphi(b^*a)^*\varphi(c)=\varphi(a)^*\varphi(bc); \\
\label{eq:DO}\tag{3.3}
& \varphi(ab)\varphi(c)^*=\varphi(a)\varphi(cb^*)^*.
\end{align}
For later use, we observe that the \emph{adjoint} $\varphi^*\colon A\to B$ of $\varphi$,
defined by $\varphi^*(a)=\varphi(a^*)^*$ for all $a\in A$, also 
preserves zero-products, range-orthogonality, and also domain-orthogonality.
We fix a representation $B\subseteq\Bdd(H)$ of $B$ on some Hilbert space.
Our goal is to construct a $\ast$-homomorphism $\pi_\varphi\colon A\to B''\subseteq\Bdd(H)$ such that
\[
\pi_\varphi(a)\varphi(b) = \varphi(ab) = \varphi(a)\pi_\varphi(b)
\]
for all $a,b\in A$.
Using this, we will show in \autoref{prp:CharWeighted} that $\varphi$ is a weighted $\ast$-homomorphism
in the sense of \autoref{dfn:WeightedStarHomo}. 

We begin the construction by setting
\[
H_0:=\mathrm{span} \left( \big\{\varphi(a)\xi:a\in A, \xi\in H\big\} \cup \big\{\varphi(a)^*\xi:a\in A, \xi\in H\big\} \right).
\]
Given $a\in A$, we first show that there exists a unique operator $\pi_\varphi(a)\in\Bdd(H)$ such that
\begin{align}\tag{3.4}
\label{eq:define-pi-a}
\pi_\varphi(a)\varphi(b)\xi = \varphi(ab)\xi, \andSep
\pi_\varphi(a)\varphi(c)^*\eta = \varphi(ca^*)^*\eta
\end{align}
for all $b,c\in A$ and $\xi,\eta\in H$, and such that $\pi_\varphi(a)\zeta=0$ for all $\zeta\in H_0^\perp$.

\begin{lma}
\label{prp:pi-a-bounded}
Let $J$ and $K$ be finite index sets, let $b_j\in A$ and $\xi_j\in H$ for $j\in J$, and let $c_k\in A$ and $\eta_k\in H$ for $k\in K$.
Then
\begin{align*}
\left\| \sum_{j\in J} \varphi(ab_j)\xi_j + \sum_{k\in K} \varphi(c_ka^*)^*\eta_k \right\| 
\leq \| a \| \left\| \sum_{j\in J} \varphi(b_{j})\xi_{j} + \sum_{k\in K} \varphi(c_{k})^*\eta_{k} \right\|
\end{align*}
for all $a\in A$.
\end{lma}
\begin{proof}
We first establish the following

\textbf{Claim 1:} \emph{Given $a\in A$, we have
\begin{align*}
&\ \ \ \ \ \ \ \ \ \ \ \ \ \ \ \ \ \ \ 
\Big\| \sum_j \varphi(ab_j)\xi_j + \sum_k \varphi(c_ka^*)^*\eta_k \Big\|^2 \\
&\quad \leq \Big\| \sum_j \varphi(a^*ab_j)\xi_j + \sum_k \varphi(c_ka^*a)^*\eta_k \Big\|
\Big\| \sum_{j} \varphi(b_{j})\xi_{j} + \sum_{k} \varphi(c_{k})^*\eta_{k} \Big\|.
\end{align*}
}

To prove the claim, let $j\in J$ and $k'\in K$. 
Then
\begin{align*}
\big\langle \varphi(ab_j)\xi_j, \varphi(c_{k'}a^*)^*\eta_{k'} \big\rangle 
&\ = \ \big\langle \varphi(c_{k'}a^*)\varphi(ab_j)\xi_j, \eta_{k'} \big\rangle \\
&\stackrel{\eqref{eq:ZP}}{=}  \big\langle \varphi(c_{k'})\varphi(a^*ab_j)\xi_j, \eta_{k'} \big\rangle \\
&\ = \ \big\langle \varphi(a^*ab_j)\xi_j, \varphi(c_{k'})^*\eta_{k'} \big\rangle.
\end{align*}

Similarly, for all $j,j'\in J$ we obtain
\begin{align*}
\big\langle \varphi(ab_j)\xi_j, \varphi(ab_{j'})\xi_{j'} \big\rangle
&\ = \ \big\langle \varphi(ab_{j'})^*\varphi(ab_j)\xi_j, \xi_{j'} \big\rangle \\
&\stackrel{\eqref{eq:RO}}{=} \big\langle \varphi(b_{j'})^*\varphi(a^*ab_j)\xi_j, \xi_{j'} \big\rangle \\
&\ = \ \big\langle \varphi(a^*ab_j)\xi_j, \varphi(b_{j'})\xi_{j'} \big\rangle.
\end{align*}

Also, for all $k,k'\in K$ we have
\begin{align*}
\big\langle \varphi(c_ka^*)^*\eta_k, \varphi(c_{k'}a^*)^*\eta_{k'} \big\rangle 
&\ = \ \big\langle \varphi(c_{k'}a^*)\varphi(c_ka^*)^*\eta_k, \eta_{k'} \big\rangle \\
&\stackrel{\eqref{eq:DO}}{=}  \big\langle \varphi(c_{k'})\varphi(c_ka^*a)^*\eta_k, \eta_{k'} \big\rangle \\
&\ = \ \big\langle \varphi(c_ka^*a)^*\eta_k, \varphi(c_{k'})^*\eta_{k'} \big\rangle.
\end{align*}

Further, for each $k\in K$ and $j'\in J$ we get
\begin{align*}
\big\langle \varphi(c_ka^*)^*\eta_k, \varphi(ab_{j'})\xi_{j'} \big\rangle
= \big\langle \varphi(c_ka^*a)^*\eta_k, \varphi(b_{j'})\xi_{j'} \big\rangle.
\end{align*}

Using all of these equalities at the third step, and using the Cauchy-Schwarz inequality at the last step, we get
\begin{align*}
&\Big\| \sum_j \varphi(ab_j)\xi_j + \sum_k \varphi(c_ka^*)^*\eta_k \Big\|^2 \\
&\quad= \left\langle \sum_j \varphi(ab_j)\xi_j + \sum_k \varphi(c_ka^*)^*\eta_k, \sum_{j'} \varphi(ab_{j'})\xi_{j'} + \sum_{k'} \varphi(c_{k'}a^*)^*\eta_{k'} \right\rangle \\
&\quad= \sum_{j,j'} \left\langle \varphi(ab_j)\xi_j, \varphi(ab_{j'})\xi_{j'} \right\rangle
+ \sum_{j,k'} \left\langle \varphi(ab_j)\xi_j, \varphi(c_{k'}a^*)^*\eta_{k'} \right\rangle \\
&\qquad\qquad + \sum_{k,j'} \left\langle \varphi(c_ka^*)^*\eta_k, \varphi(ab_{j'})\xi_{j'} \right\rangle
+ \sum_{k,k'} \left\langle \varphi(c_ka^*)^*\eta_k, \varphi(c_{k'}a^*)^*\eta_{k'} \right\rangle \\
&\quad= \sum_{j,j'} \left\langle \varphi(a^*ab_j)\xi_j, \varphi(b_{j'})\xi_{j'} \right\rangle
+ \sum_{j,k'} \left\langle \varphi(a^*ab_j)\xi_j, \varphi(c_{k'})^*\eta_{k'} \right\rangle \\
&\qquad\qquad + \sum_{k,j'} \left\langle \varphi(c_ka^*a)^*\eta_k, \varphi(b_{j'})\xi_{j'} \right\rangle
+ \sum_{k,k'} \left\langle \varphi(c_ka^*a)^*\eta_k, \varphi(c_{k'})^*\eta_{k'} \right\rangle \\
&\quad= \left\langle \sum_j \varphi(a^*ab_j)\xi_j + \sum_k \varphi(c_ka^*a)^*\eta_k, \sum_{j'} \varphi(b_{j'})\xi_{j'} + \sum_{k'} \varphi(c_{k'})^*\eta_{k'} \right\rangle \\
&\quad\leq \Big\| \sum_j \varphi(a^*ab_j)\xi_j + \sum_k \varphi(c_ka^*a)^*\eta_k \Big\|
\Big\| \sum_{j} \varphi(b_{j})\xi_{j} + \sum_{k} \varphi(c_{k})^*\eta_{k} \Big\|.
\end{align*}
This proves the claim.

Fix $a\in A$ for the rest of the proof. We set
\[
C= \Big\| \sum_j \varphi(ab_j)\xi_j + \sum_k \varphi(c_ka^*)^*\eta_k \Big\| \ \text{and}\
D= \Big\| \sum_{j} \varphi(b_{j})\xi_{j} + \sum_{k} \varphi(c_{k})^*\eta_{k} \Big\|.
\]
For $n\in\mathbb{N}$, we denote by $I(n)$ the (potentially false) expression
\begin{align*}
\label{eq:pi-a-bounded:estimte}\tag*{$I(n)$}
C^2 \leq \Big\| \sum_j \varphi((a^*a)^{2^n}b_j)\xi_j + \sum_k \varphi(c_ka^*a)^*\eta_k \Big\|^{\tfrac{1}{2^n}}
D^{2-\tfrac{1}{2^n}}.
\end{align*}

\textbf{Claim~2:} \emph{$I(n)$ is true for all $n\in\mathbb{N}$.}
We will prove the claim by induction.
Applying Claim~1, we get
\begin{align*}
&C^2  = \left\| \sum_j \varphi(ab_j)\xi_j + \sum_k \varphi(c_ka^*)^*\eta_k \right\|^2 \\
&\quad \leq \left\| \sum_j \varphi(a^*ab_j)\xi_j + \sum_k \varphi(c_ka^*a)^*\eta_k \right\|
\left\| \sum_{j} \varphi(b_{j})\xi_{j} + \sum_{k} \varphi(c_{k})^*\eta_{k} \right\| \\
&\quad  = \left\| \sum_j \varphi(a^*ab_j)\xi_j + \sum_k \varphi(c_ka^*a)^*\eta_k \right\|^{\tfrac{1}{2^0}} D^{2-\tfrac{1}{2^0}}.
\end{align*}
In other words, $I(0)$ holds.
For the induction step, assume that $I(n)$ is true for some $n\geq 0$.
Applying Claim~1 with the self-adjoint element $(a^*a)^{2^n}$ in place of $a$, we obtain
\begin{align*}
&\left\| \sum_j \varphi((a^*a)^{2^n}b_j)\xi_j + \sum_k \varphi(c_k(a^*a)^{2^n})^*\eta_k \right\| \\
&\qquad \leq \left\| \sum_j \varphi((a^*a)^{2^{n+1}}b_j)\xi_j + \sum_k \varphi(c_k(a^*a)^{2^{n+1}})^*\eta_k \right\|^{\tfrac{1}{2}} D^{\tfrac{1}{2}}.
\end{align*}
Using this at the second step, we obtain
\begin{align*}
C^2 &\stackrel{I(n)}{\leq} \left\| \sum_j \varphi((a^*a)^{2^n}b_j)\xi_j + \sum_k \varphi(c_k(a^*a)^{2^n})^*\eta_k \right\|^{\tfrac{1}{2^n}} D^{2-\tfrac{1}{2^n}} \\
&\leq \left\| \sum_j \varphi((a^*a)^{2^{n+1}}b_j)\xi_j + \sum_k \varphi(c_k(a^*a)^{2^{n+1}})^*\eta_k \right\|^{\tfrac{1}{2^{n+1}}} D^{\tfrac{1}{2^{n+1}}}D^{2-\tfrac{1}{2^n}},
\end{align*} 
which shows that $I(n+1)$ is true. 
This completes the induction and proves the claim.

Now, for every $n\geq 0$, we get
\begin{align*}
C^2 &\stackrel{I(n)}{\leq} \left\| \sum_j \varphi((a^*a)^{2^n}b_j)\xi_j + \sum_k \varphi(c_ka^*a)^*\eta_k \right\|^{\tfrac{1}{2^n}} D^{2-\tfrac{1}{2^n}} \\
&\leq \left( \sum_j \|\varphi\|\|\xi_j\| + \sum_k \|\varphi\|\|\eta_k\| \right)^{\tfrac{1}{2^n}} 
\|(a^*a)^{2^n}\|^{\tfrac{1}{2^n}} D^{2-\tfrac{1}{2^n}}.
\end{align*}
Using that $\|(a^*a)^{2^n}\|^{\tfrac{1}{2^n}}=\|a\|^2$, and taking to the limit as $n\to\infty$, we deduce that $C^2 \leq \|a\|^2 D^2$, and so $C\leq\|a\|D$, as desired.
\end{proof}

Given $a\in A$, it follows from \autoref{prp:pi-a-bounded} that the map $H_0\to H_0$, given by
\[
\sum_j\varphi(b_j)\xi_j + \sum_k\varphi(c_k)^*\eta_k 
\mapsto \sum_j \varphi(ab_j)\xi_j + \sum_k \varphi(c_ka^*)^*\eta_k
\]
is well-defined, bounded and linear, and therefore extends to a bounded, linear map $\pi^{(0)}_\varphi(a)\colon \overline{H_0}\to\overline{H_0}$ satisfying $\|\pi^{(0)}_\varphi(a)\|\leq \|a\|$ for all $a\in A$.
We define $\pi_\varphi(a)\in\Bdd(H)$ by $\pi_\varphi(a)\xi:=\pi^{(0)}_\varphi(a)\xi$ for $\xi\in\overline{H_0}$ and $\pi_\varphi(a)\eta:=0$ for $\eta\in\overline{H_0}^\perp$.
We note that $\pi_\varphi$ is determined by
\begin{align}\label{eqn:Defpi}\tag{3.5}
\pi_\varphi(a)\varphi(b)\xi=\varphi(ab)\xi \ \ \mbox{ and } \ \ 
 \pi_\varphi(a)\varphi(b)^*\xi=\varphi(ba^*)^*\xi 
\end{align}
for all $a,b\in A$ and all $\xi\in H$. Moreover, one readily checks that $\pi_\varphi=\pi_{\varphi^*}$.

\begin{prp}
\label{prp:formulasPiPhiSmall}
Let $A$ and $B$ be \ca s with $B\subseteq \Bdd(H)$, and let $\varphi\colon A\to B$ be a bounded, linear map between \ca{s} that preserves zero-products, range-orthogonality, and also domain-orthogonality.
Denote by $\pi_\varphi\colon A\to\Bdd(H)$ the canonical bounded, linear map defined in the preceding comments. 
Then:
\begin{enumerate} 
\item 
The map $\pi_\varphi$ is a $\ast$-homomorphism.
\item 
The image of $\pi_\varphi$ is contained in $B''\subseteq \Bdd(H)$.
\item 
For all $a,b\in A$, we have
\[
\pi_\varphi(a)\varphi(b)=\varphi(ab)=\varphi(a)\pi_\varphi(b) \andSep
\pi_\varphi(a)\varphi(b)^*=\varphi(ba^*)^*=\varphi(a^*)^*\pi_\varphi(b^*).
\]
\end{enumerate}
\end{prp}
\begin{proof} 
(1)
We first show that $\pi_\varphi$, which we will abbreviate to $\pi$ throughout in the proof of this proposition, is linear and multiplicative.
For $a_1,a_2$ and $\lambda\in\mathbb{C}$, we have
\begin{align*}
\pi(a_1+\lambda a_2)\varphi(b)\xi
&= \varphi\big( (a_1+\lambda a_2)b \big)\xi \\
&= \varphi(a_1b)\xi + \lambda \varphi(a_2b)\xi
= \big( \pi(a_1)+\lambda\pi(a_2) \big)\varphi(b)\xi
\end{align*}
and
\[
\pi(a_1)\pi(a_2)\varphi(b)\xi
= \pi(a_1)\varphi(a_2b)\xi
= \varphi(a_1a_2b)\xi
= \pi(a_1a_2)\varphi(b)\xi
\]
for all $b\in A$ and $\xi\in H$.
Similarly, we have
\begin{align*}
\pi(a_1+\lambda a_2)\varphi(c)^*\eta
&= \varphi( c(a_1+\lambda a_2)^* )^*\eta \\
&= \varphi(ca_1^*)^*\eta + \lambda \varphi(ca_2^*)^*\eta
= \big( \pi(a_1)+\lambda\pi(a_2) \big) \varphi(c)^*\eta
\end{align*}
and
\[
\pi(a_1)\pi(a_2)\varphi(c)^*\eta
= \pi(a_1)\varphi(ca_2^*)^*\eta
= \varphi(ba_2^*a_1^*)^*\eta
= \pi(a_1a_2)\varphi(c)^*\eta
\]
for all $c\in A$ and $\eta\in H$.
Using linearity and continuity of $\pi(a_1+\lambda a_2)$ and $\pi(a_1)+\lambda\pi(a_2)$, it follows that these operators agree on $\overline{H_0}$.
Since both operators vanish on $\overline{H_0}^\perp$, it follows that $\pi(a_1+\lambda a_2)=\pi(a_1)+\lambda\pi(a_2)$.
Similarly, we obtain $\pi(a_1)\pi(a_2)=\pi(a_1a_2)$.
Thus, $\pi$ is linear and multiplicative.

To show that $\pi$ preserves adjoints, let $a\in A$.
Given $b,b'\in A$ and $\xi,\xi'\in H$, using \eqref{eq:RO} at the fourth step, we get
\begin{align*}
\big\langle \pi(a)^*\varphi(b)\xi, \varphi(b')\xi' \big\rangle
&= \big\langle \varphi(b)\xi, \pi(a)\varphi(b')\xi' \big\rangle 
= \big\langle \varphi(b)\xi, \varphi(ab')\xi' \big\rangle \\
&= \big\langle \varphi(ab')^*\varphi(b)\xi, \xi' \big\rangle 
= \big\langle \varphi(b')^*\varphi(a^*b)\xi, \xi' \big\rangle \\
&= \big\langle \varphi(a^*b)\xi, \varphi(b')\xi' \big\rangle 
= \big\langle \pi(a^*)\varphi(b)\xi, \varphi(b')\xi' \big\rangle.
\end{align*}
Similarly, given $c,c'\in A$ and $\eta,\eta'\in H$, using \eqref{eq:DO} at the fourth step, we get
\begin{align*}
\big\langle \pi(a)^*\varphi(c)^*\eta, \varphi(c')^*\eta' \big\rangle
&= \big\langle \varphi(c)^*\eta, \pi(a)\varphi(c')^*\eta' \big\rangle 
= \big\langle \varphi(c)^*\eta, \varphi(c'a^*)^*\eta' \big\rangle \\
&= \big\langle \varphi(c'a^*)\varphi(c)^*\eta, \eta' \big\rangle 
= \big\langle \varphi(c')\varphi(ca)^*\eta, \eta' \big\rangle \\
&= \big\langle \varphi(ca)^*\eta, \varphi(c')^*\eta' \big\rangle 
= \big\langle \pi(a^*)\varphi(c)^*\eta, \varphi(c')^*\eta' \big\rangle.
\end{align*}

Further, given $b,c\in A$ and $\xi,\eta\in H$, using \eqref{eq:ZP} at the fourth step, we get
\begin{align*}
\big\langle \pi(a)^*\varphi(b)\xi, \varphi(c)^*\eta \big\rangle
&= \big\langle \varphi(b)\xi, \pi(a)\varphi(c)^*\eta \big\rangle 
= \big\langle \varphi(b)\xi, \varphi(ca^*)^*\eta \big\rangle \\
&= \big\langle \varphi(ca^*)\varphi(b)\xi, \eta \big\rangle 
= \big\langle \varphi(c)^*\varphi(a^*b)\xi, \eta \big\rangle \\
&= \big\langle \varphi(a^*b)\xi, \varphi(c)^*\eta \big\rangle 
= \big\langle \pi(a^*)\varphi(b)\xi, \varphi(c)^*\eta \big\rangle.
\end{align*}
and analogously
\[
\big\langle \pi(a)^*\varphi(c)^*\eta, \varphi(b)\xi \big\rangle
= \big\langle \pi(a^*)\varphi(c)^*\eta, \varphi(b)\xi \big\rangle.
\]

Using linearity and continuity of $\pi(a)^*$ and $\pi(a^*)$, we get
\[
\big\langle \pi(a)^* \alpha, \beta \big\rangle
= \big\langle \pi(a^*)\alpha, \beta\big\rangle
\]
for all $\alpha,\beta\in\overline{H_0}$ and consequently $\pi(a)^*=\pi(a^*)$, as desired.

(2)
Let $a\in A$, and let $x\in B'$.
We need to show that $\pi(a)x=x\pi(a)$.
Given $b\in B$ and $\xi\in H$, using at the first and third step that $x$ commutes with $\varphi(b)$ and~$\varphi(ab)$, we have
\[
\pi(a)x\varphi(b)\xi
= \pi(a)\varphi(b)x\xi
= \varphi(ab)x\xi
= x\varphi(ab)\xi
= x\pi(a)\varphi(b)\xi.
\]
Similarly, given $c\in B$ and $\xi\in H$, using that $\varphi(c)^*$ and $\varphi(ca^*)^*$ belong to $B$ and that they therefore commute with $x$, we get
\[
\pi(a)x\varphi(c)^*\xi
= \pi(a)\varphi(c)^*x\xi
= \varphi(ca^*)^*x\xi
= x\varphi(ca^*)^*\xi
= x\pi(a)\varphi(c)^*\xi.
\]
Using linearity and continuity of $\pi(a)x$ and $x\pi(a)$, it follows that $\pi(a)x\xi=x\pi(a)\xi$ for all $\xi\in\overline{H_0}$.

Using that $x$ commutes with $\varphi(b)$ and $\varphi(c)^*$ for all $b,c\in A$, one readily 
checks that $x(\overline{H_0})\subseteq\overline{H_0}$.
Similarly, since $x^*$ commutes with $\varphi(b)$ and $\varphi(c)^*$ for all $b,c\in A$, we get $x^*(\overline{H_0})\subseteq\overline{H_0}$.
It follows that $x$ leaves $\overline{H_0}^\perp$ invariant, and thus
\[
\pi(a)x\eta = 0 = x\pi(a)\eta
\]
for all $\eta\in\overline{H_0}^\perp$.
In conclusion, we obtain $\pi(a)x=x\pi(a)$, as desired.

(3)
Let $a,b\in A$.
We have
\[
\pi(a)\varphi(b)\xi\stackrel{\eqref{eqn:Defpi}}{=}\varphi(ab)\xi
\]
for all $\xi\in H$, and therefore $\pi(a)\varphi(b)=\varphi(ab)$.

To check that $\varphi(a)\pi(b)=\varphi(ab)$, we prove that these
operators agree both on $\overline{H_0}$ and on its orthogonal complement. 
Given $c\in A$, we have
\begin{align}\label{eq:Comm1}\tag{3.6}
\varphi(a)\pi(b)\varphi(c)\xi\stackrel{\eqref{eqn:Defpi}}{=}\varphi(a)\varphi(bc)\xi
\stackrel{\eqref{eq:ZP}}{=}\varphi(ab)\varphi(c)\xi. 
\end{align}
Similarly,
\begin{align}\label{eq:Comm2}\tag{3.7}
\varphi(a)\pi(b)\varphi(c)^*\xi\stackrel{\eqref{eqn:Defpi}}{=}\varphi(a)\varphi(cb^*)^*\xi
\stackrel{\eqref{eq:DO}}{=}\varphi(ab)\varphi(c)^*\xi. 
\end{align}
It follows from \eqref{eq:Comm1} and \eqref{eq:Comm2} that the linear maps $\varphi(a)\pi(b)$ and $\varphi(ab)$ agree on~$\overline{H_0}$.
Let $\eta\in\overline{H_0}^\perp$. Then $\pi(b)\eta=0$. On the other hand,
we have
\[\langle \varphi(ab)\eta,\varphi(ab)\eta\rangle=
 \langle \eta,\varphi(ab)^*\varphi(ab)\eta\rangle=0,
\]
where at the last step we use that $\varphi(ab)^*\varphi(ab)\eta$ belongs to $\overline{H_0}$. 
This shows that~$\varphi(ab)$ and~$\varphi(a)\pi(b)$ agree on $\overline{H_0}^\perp$, and hence 
$\varphi(a)\pi(b)=\varphi(ab)$.

The last two equalities are proved analogously. (They also follow by 
replacing $\varphi$ with $\varphi^*$,
which satisfies the same assumptions as $\varphi$ and has
$\pi_\varphi=\pi_{\varphi^*}$.) 
\end{proof}

\begin{ntn}
\label{pgr:extensions}
Let $A$ and $B$ be \ca s and let $\varphi\colon A\to B\subseteq \Bdd(H)$ be a bounded, linear map that preserves zero-products, range-orthogonality, and also domain-orthogonality.
Denote by $\pi_\varphi\colon A\to B''$ the canonical $\ast$-homomorphism provided by
\autoref{prp:formulasPiPhiSmall}.
We denote by $\Phi,\Pi_\varphi\colon A^{**}\to\Bdd(H)$ the (unique) extensions of $\varphi,\pi_\varphi\colon A\to\Bdd(H)$ to weak*-continuous, bounded linear maps,
which are explicitly constructed as follows.
Let $S_1(H)$ denote the space of trace-class of operators on $H$, which we naturally identify with the (unique) isometric predual of $\Bdd(H)$.
Let $\kappa\colon S_1(H)\to S_1(H)^{**}$ be the natural inclusion of $S_1(H)$ into its bidual.
Then the transpose map $\kappa^*\colon \Bdd(H)^{**}\cong S_1(H)^{***}\to S_1(H)^{*}\cong\Bdd(H)$ is a weak*-continuous $\ast$-ho\-mo\-mor\-phism, and
\[
\Phi = \kappa^*\circ\varphi^{**}\colon A^{**}\to \Bdd(H), \andSep
\Pi_\varphi = \kappa^*\circ\pi_\varphi^{**}\colon A^{**}\to \Bdd(H).
\]
By \autoref{prp:Bitranspose}, $\varphi^{**}$ preserves range-orthogonality, domain-orthogonality and zero-products, and hence so does $\Phi$.
Similarly, $\Pi_\varphi$ is a $\ast$-homomorphism.
Note that the images of $\Phi$ and $\Pi_\varphi$ are contained in $B''\subseteq\Bdd(H)$.
\end{ntn}

\begin{prp}
\label{prp:formulasPiPhiBig}
Let $A$ and $B$ be \ca s and let 
$\varphi\colon A\to B\subseteq \Bdd(H)$ be a bounded, linear map that preserves zero-products, range-orthogonality, and also domain-orthogonality.
Denote by $\pi_\varphi\colon A\to B''$ the canonical $\ast$-homomorphism provided by
\autoref{prp:formulasPiPhiSmall}, and let 
$\Phi,\Pi_\varphi\colon A^{\ast\ast}\to B''$ the maps from \autoref{pgr:extensions}.
\begin{enumerate}
\item 
For all $a,b\in A$, we have
\[
\Pi_\varphi(a)\Phi(b)=\Phi(ab)=\Phi(a)\Pi_\varphi(b) \andSep
\Pi_\varphi(a)\Phi(b)^*=\Phi(ba^*)^*=\Phi(a^*)^*\Pi_\varphi(b^*).
\]
\end{enumerate}
Let $C$ be the (not necessarily self-adjoint) closed subalgebra of $B$ generated by the image of $\varphi$, and let $D:=C^*(\varphi(A))$ be the sub-\ca{} of $B$ generated by $\varphi(A)$.
\begin{enumerate}
\setcounter{enumi}{1}
\item 
Let $a\in M(A)\subseteq A^{**}$.
Then $\Phi(a)$ and $\Pi_\varphi(a)$ normalize $C$ and D, that is, $\Phi(a)x,x\Phi(a),\Pi_\varphi(a)x$, and $x\Pi_\varphi(a)$ belong to $C$ (respectively, to $D$) for every $x\in C$ (respectively, $x\in D$).
\item 
If $a$ belongs to the center of $M(A)$, then $\Phi(a)$ belongs to $C'$.
\end{enumerate}
\end{prp}
\begin{proof}
(1)
This follows from part~(3) of~\autoref{prp:formulasPiPhiSmall} by applying the same argument as in the proof of \autoref{prp:Bitranspose}, using twice that multiplication on $\Bdd(H)$ is separately
weak*-continuous.

(2)
We only show the normalization for $D$.
The proof for $C$ is similar, but easier.
Set
\[
G := \{ \varphi(b) \colon b\in A\} \cup \{ \varphi(c)^*\colon c\in A \}.
\]
Then the set of finite linear sums of finite products of elements in $G$ is dense in~$D$.

We first show that $\Pi_\varphi(a)$ is a left normalizer of $D$, that is, $\Pi_\varphi(a)D\subseteq D$.
Using that $\Pi_\varphi(a)$ is a bounded, linear operator, it suffices to show that $\Pi_\varphi(a)x\in D$ whenever 
$x\in G^n$ is a finite product of elements in $G$.
Let $x=gy$ with $g\in G$ and $y\in G^n$, for some $n\in\mathbb{N}$.
If $g=\varphi(b)$ for some $b\in A$, then applying part~(1) of~\autoref{prp:formulasPiPhiBig}, and using that $ab\in A$, we obtain
\[
\Pi_\varphi(a)x
= \Pi_\varphi(a)gy
= \Pi_\varphi(a)\Phi(b)y
= \Phi(ab)y 
= \varphi(ab)y 
\in G^{n+1}\subseteq D.
\]
Similarly, if $g=\varphi(c)^*$ for some $c\in A$, using that $ca^*\in A$, we obtain
\[
\Pi_\varphi(a)x
= \Pi_\varphi(a)gy
= \Pi_\varphi(a)\Phi(c)^*y
= \Phi(ca^*)^*y
= \varphi(ca^*)^*y 
\in G^{n+1}\subseteq D.
\]

Analogously, one verifies that $\Pi_\varphi(a)$ is a right normalizer of $D$, that is, we have $D\Pi_\varphi(a)\subseteq D$.

Next, we show that $\Phi(a)$ is a left normalizer of $D$.
Again, it suffices to verify that $\Phi(a)x\in D$ for $x=gy$ with $g\in G$ and $y\in G^n$ 
a product of $n$ elements in $G$.
If $g=\varphi(b)$ for some $b\in A$, then choose $b_1,b_2\in A$ with $b=b_1b_2$.
Using that $\Phi$ preserves zero-products and therefore satisfies the formula in \autoref{prp:CharOrthPres}(1), and using that $ab_1\in A$, we obtain
\[
\Phi(a)x
= \Phi(a)gy
= \Phi(a)\Phi(b_1b_2)y
= \Phi(ab_1)\Phi(b_2)y 
= \varphi(ab_1)\varphi(b_2)y 
\in G^{n+2}\subseteq D.
\]

Similarly, if $g=\varphi(c)^*$ for some $c\in A$, then choose $c_1,c_2\in A$ with $c=c_1c_2$.
Using that $\Phi$ preserves domain-orthogonality and therefore satisfies the formula in
part~(3) of~\autoref{prp:CharOrthPres}, and using that $ac_2^*\in A$, we obtain
\[
\Phi(a)x
= \Phi(a)gy
= \Phi(a)\Phi(c_1c_2)^*y
= \Phi(ac_2^*)\Phi(c_1)^*y
= \varphi(ac_2^*)\varphi(c_1)^*y 
\in G^{n+2}\subseteq D.
\]
Analogously, one verifies that $\Phi(a)$ is a left normalizer for $D$.

(3)
Let $a\in M(A)$ be central.
Let $b\in A$ and denote by $1_A$ the unit of $M(A)$.
Using that $\Phi$ preserves zero-products and therefore satisfies the formula in \autoref{prp:CharOrthPres}(1), we get
\[
\Phi(a)\varphi(b)
= \Phi(1_Aa)\Phi(b)
= \Phi(1_A)\Phi(ab)
= \Phi(1_A)\Phi(ba)
= \Phi(1_Ab)\Phi(a)
= \varphi(b)\Phi(a).
\]
We deduce that $\Phi(a)$ commutes with every finite linear combination of products of finitely many elements in $\varphi(A)$.
Since such elements are dense in $C$, we have $\Phi(a)\in C'$.
\end{proof}

\section{Weighted \texorpdfstring{$\ast$}{*}-homomorphisms}
\label{sec:weightedHomo}

In this section, we prove the main result of the paper (\autoref{prp:CharWeighted}), where we obtain an intrinsic and algebraic characterization of weighted $\ast$-homomorphisms between \ca{s}.
Our characterization simplifies considerably for self-adjoint and positive maps; 
see Corollaries~\ref{prp:CharWeightedSelfadjoint} and~\ref{prp:CharOrderZero}.

\begin{lma}
\label{prp:CanonicalSupportHomo}
Let $\varphi\colon A\to B$ be a bounded, linear map between \ca{s} that preserves range-orthogonality, domain-orthogonality and zero-products.
Assume that $B=C^*(\varphi(A))$, that is, $\varphi(A)$ is not contained in a proper sub-\ca{} of $B$.
Let~$C$ be the closed subalgebra of $B$ generated by the image of~$\varphi$.
Set $h:=\varphi^{**}(1)\in B^{**}$.

Then there exists a canonical $\ast$-homomorphism $\pi_\varphi\colon A\to B^{**}$ such that $\varphi(a)=h\pi_\varphi(a)=\pi_\varphi(a)h$ for all $a\in A$, and such that $\pi_\varphi(a)$ normalizes both $B$ and $C$ for every $a\in A$.
Further, $h$ normalizes $B$ and $C$, and commutes with every element of~$C$.
In particular, we may view $\pi_\varphi$ as a $\ast$-homomorphism $\pi_\varphi\colon A\to M(B)\cap\{h\}'$, and $h$ belongs to $M(B)\cap C'$.
\end{lma}
\begin{proof}
Use \cite[Section~III.5.2]{Bla06OpAlgs} to
choose a (universal) representation $B\subseteq\Bdd(H)$ such that,
with $\kappa^*\colon\Bdd(H)^{**}\to\Bdd(H)$ denoting the $\ast$-homomorphism described in \autoref{pgr:extensions}, the restriction of $\kappa^*$ to $B^{**}$ is an isomorphism onto $B''$.
The situation is shown in the following commutative diagram:
\[
\xymatrix{
& B^{**} \ar@{^{(}->}[rr] \ar[dr]_{\cong}^{\kappa^*|_{B^{**}}}
&& \Bdd(H)^{**} \ar[d]^{\kappa^*} \\
A \ar[r]^{\varphi} \ar[ur]^{\varphi^{**}}
& B \ar@{^{(}->}[r] \ar@{^{(}->}[u]
& B'' \ar@{^{(}->}[r]
& \Bdd(H).
}
\]

By \autoref{prp:formulasPiPhiSmall}, 
there is a canonical *-ho\-mo\-mor\-phism $\pi_\varphi\colon A\to B''$ such that
\[
\pi_\varphi(a)\varphi(b) = \varphi(ab)=\varphi(a)\pi_\varphi(b)
\]
for all $a,b\in A$.
Let $\Phi=\varphi^{**}\colon A^{**}\to B^{**}\subseteq \Bdd(H)$ be the unique extension of
$\varphi$ to a weak-* continuous map.
After identifying $B^{**}$ with $B''$, the map $\Phi$ is simply the bitranspose of $\varphi$.
Hence, applying part~(1) of~\autoref{prp:formulasPiPhiBig} at the second steps, we obtain
\begin{align*}
\varphi(a)
&= \Phi(1a)
= \Phi(1)\pi_\varphi(a)
= h\pi_\varphi(a), \andSep \\
\varphi(a)
&= \Phi(a1)
= \pi_\varphi(a)\Phi(1)
= \pi_\varphi(a)h
\end{align*}
for every $a\in A$.

It follows from part~(2) of~\autoref{prp:formulasPiPhiBig} that $\pi_\varphi(a)$ normalizes both~$B$ and~$C$, for each $a\in A$.
Moreover, $h=\Phi(1)$ normalizes both $B$ and $C$ by
part~(3) of~\autoref{prp:formulasPiPhiBig}, and clearly belongs to $C'$.
\end{proof}

The following theorem characterizes weighted $\ast$-homomorphisms (\autoref{dfn:WeightedStarHomo}) 
in terms of orthogonality-preservation
properties. Note that part (2) gives canonical choices for the algebra $D$ and $h\in D$. 

\begin{thm}
\label{prp:CharWeighted}
Let $\varphi\colon A\to B$ be a bounded, linear map between \ca{s}.
Then the following are equivalent:
\begin{enumerate}
\item
$\varphi$ is a weighted $\ast$-homomorphism, namely:
there exist a \ca{} $D$ with $B\subseteq D$, a $\ast$-homomorphism $\pi\colon A\to D$ and $h\in D$ such that $\varphi(a)=h\pi(a)=\pi(a)h$ for all $a\in A$;
\item
there exists a (canonical) $\ast$-homomorphism $\pi_\varphi\colon A\to B^{**}$ such that we have $\varphi(a)=\varphi^{**}(1)\pi_\varphi(a)=\pi_\varphi(a)\varphi^{**}(1)$ for all $a\in A$;
\item
$\varphi$ preserves range-orthogonality, domain-orthogonality and zero-products.
\end{enumerate}
\end{thm}
\begin{proof}
It is clear that~(2) implies~(1).
To show that~(1) implies~(3), assume that there exist a \ca{} $D$ with $B\subseteq D$, a $\ast$-homomorphism $\pi\colon A\to D$ and $h\in D$ such that $\varphi(a)=h\pi(a)=\pi(a)h$ for all $a\in A$.
To verify that $\varphi$ preserves range-orthogonality, let $a,b\in A$ satisfy $a^*b=0$.
Then
\[
\varphi(a)^*\varphi(b)
= (\pi(a)h)^*(\pi(b)h)
= h^*\pi(a)^*\pi(b)h
= h^*\pi(a^*b)h
= 0.
\]
Similarly, one verifies that $\varphi$ preserves domain-orthogonality and zero-products.

To show that~(3) implies~(2), assume that $\varphi$ preserves range-orthogonality, domain-orthogonality and zero-products.
Let $B_0:=C^*(\varphi(A))$ be the sub-\ca{} of $B$ generated by the image of $\varphi$.
Let $\varphi_0\colon A\to B_0$ be the corestriction, and set $h_0:=\varphi_0^{**}(1)\in B_0^{**}$.
Applying \autoref{prp:CanonicalSupportHomo}, we obtain a $\ast$-homomorphism $\pi_0\colon A\to B_0^{**}$ such that $\varphi_0(a)=h_0\pi_0(a)=\pi_0(a)h_0$ for all $a\in A$.
The inclusion $\iota\colon B_0\hookrightarrow B$ induces a natural inclusion $\iota^{**}\colon B_0^{**}\hookrightarrow B^{**}$ that identifies $h_0$ with $\varphi^{**}(1)$.
Then the $\ast$-homomorphism $\pi_\varphi:=\iota^{**}\circ\pi_0\colon A\to B^{**}$ has the desired properties.
\end{proof}

\begin{rmk}
Let us clarify the relationship between our notion of `weighted $\ast$-homomorphism' from \autoref{dfn:WeightedStarHomo} and the concept of `weighted homomorphism' used in the theory of (Banach) algebras.

A map $\varphi\colon A\to B$ between Banach algebras is said to be a \emph{weighted homomorphism} if there exists a homomorphism $\pi\colon A\to B$ and an invertible centralizer $W$ on $B$ such that $\varphi=W\pi$.
Here, a \emph{centralizer} on $B$ is a linear map $W\colon B\to B$ satisfying $W(ab)=aW(b)=W(a)b$ for all $a,b\in B$.
We use $\Gamma(B)$ to denote the algebra of centralizers on $B$.

If $B$ is \emph{faithful} (that is, every element $b\in B$ satisfying $bB=\{0\}$ or $Bb=\{0\}$ is zero), then centralizers correspond to central multipliers:
A \emph{multiplier} (also called a \emph{double centralizer}) on $B$ is a pair $(L,R)$ of linear maps $L,R\colon B\to B$ such that $aL(b)=R(a)b$ for all $a,b\in B$.
We obtain a natural map from $\Gamma(B)$ to the multiplier algebra $M(B)$ given by $W\mapsto (W,W)$.
If $B$ is faithful, then this map defines an isomorphism between $\Gamma(B)$ and $Z(M(B))$, the center of the multiplier algebra.

Now let $\varphi\colon A\to B$ be a weighted $\ast$-homomorphism between \ca s.
Let $B_0:=C^*(\varphi(A))$ be the sub-\ca{} of $B$ generated by the image of $\varphi$, and let $C$ denote the closed subalgebra of $B$ generated by $\varphi(A)$.
By \autoref{prp:CanonicalSupportHomo}, there exists a $\ast$-homomorphism $\pi\colon A\to M(B_0)$ and a multiplier $h\in M(B_0)\cap C'$ such that $\varphi(a)=h\pi(a)=\pi(a)h$ for all $a\in A$.
However, while $h$ commutes with elements in the image of $\varphi$ (and as a consequence belongs to $C'$), it may not commute with the \emph{adjoints} of the elements in $\varphi(A)$, and therefore does not necessarily belong to $Z(M(B_0))$. 

If $h$ is normal (which is automatically the case if $\varphi$ is self-adjoint), then Fuglede's theorem implies that $h$ also commutes with the adjoints of elements in $\varphi(A)$, and consequently belongs to $M(B_0)\cap B_0'$, and therefore to $Z(M(B_0))$.
Similarly, if $\varphi(A)$ is a self-adjoint subset of $B$ (which is automatically the case if $\varphi$ is surjective), then~$h$ belongs to $Z(M(B_0))$.
In both cases, we see that the map $\varphi\colon A\to B_0\subseteq M(B_0)$ is a weighted homomorphism in the algebraic sense.

On the other hand, we may view $\varphi$ as a map from $A$ to the Banach algebra $C$.
Note that multiplication by $h$ defines a centralizer on $C$, and that $\pi(a)$ defines a multiplier on $C$ for each $a\in A$.
However, $C$ can be very pathological.
For example, if $h^2=0$, then multiplication by $h$ defines the zero centralizer on $C$, the product of any two elements in $C$ is zero, $C$ is not faithful, and the canonical map $C\to M(C)$ is the zero map.
\end{rmk}

\begin{cor}
\label{prp:CharWeightedSelfadjoint}
Let $\varphi\colon A\to B$ be a self-adjoint, bounded, linear map between \ca{s}.
Then $\varphi$ is a weighted $\ast$-homomorphism if and only if $\varphi$ preserves zero-products.
\end{cor}
\begin{proof}
Assume that $\varphi$ preserves zero-products.
For $a,b\in A$ with $a^*b=0$, we have
\[
\varphi(a)^*\varphi(b)
= \varphi(a^*)\varphi(b)
= 0.
\]
That is, $\varphi$ preserves range-orthogonality.
One similarly shows that $\varphi$ preserves domain-orthogonality, and 
the statement thus follows from \autoref{prp:CharWeighted}.
\end{proof}

The above corollary motivates the following problem.

\begin{pbm}
Characterize weighted $\ast$-homomorphisms with normal weights.
\end{pbm}

The next result is a generalization to the setting of self-adjoint, zero-product preserving maps of the structure theorem of Winter and Zacharias \cite[Theorem~3.3]{WinZac09CpOrd0} for completely positive, order zero maps.

\begin{cor}
\label{prp:StructureWeightedSelfadjoint}
Let $\varphi\colon A\to B$ be a self-adjoint, bounded, linear map between \ca{s} that preserves zero-products.
Set $C:=C^*(\varphi(A))\subseteq B$. 
Then $h:=\varphi^{**}(1)$ belongs to $M(C)\cap C'\subseteq C^{**}\subseteq B^{**}$, and there exists a canonical 
$\ast$-homomorphism $\pi_\varphi\colon A\to M(C)\cap\{h\}'$ 
such that $\varphi(a)=h\pi_\varphi(a)$ for all $a\in A$.
\end{cor}
\begin{proof}
By \autoref{prp:CharWeightedSelfadjoint}, $\varphi$ is a weighted $\ast$-homomorphism and thus preserves range-orthogonality, domain-orthogonality and zero-products by \autoref{prp:CharWeighted}.
Since~$\varphi$ is self-adjoint, it follows that $C$ agrees with the closed subalgebra of $B$ generated by~$\varphi(A)$.
Now it follows from \autoref{prp:CanonicalSupportHomo} that there exists a exists a $\ast$-ho\-mo\-mor\-phism $\pi_\varphi\colon A\to M(C)\cap\{h\}'$ such that $\varphi(a)=h\pi_\varphi(a)$ for all $a\in A$, and such that $h$ belongs to $M(C)\cap C'$.
\end{proof}

We stress the fact that the map $\pi_\varphi$ that we obtain in the 
previous corollary is natural. This has the following consequence:

\begin{rmk}
Adopt the notation and assumptions of \autoref{prp:StructureWeightedSelfadjoint}. Let $G$ be a topological
group, let $\alpha\colon G\to\mathrm{Aut}(A)$ and $\beta\colon G\to\mathrm{Aut}(B)$ be continuous actions, and suppose that
the map $\varphi\colon A\to B$ is equivariant. Then $h$ is $G$-invariant,
and there is a canonical
continuous action of $G$ on $C$, and hence there is a (not necessarily
continuous) action $\gamma\colon G\to \mathrm{Aut}(M(C)\cap \{h\}')$.
Using naturality of $\pi_\varphi$ at the first and third steps, and equivariance of $\varphi$ at the second step, we get
\[\pi_{\varphi}\circ\alpha_g=\pi_{\varphi\circ\alpha_g}=\pi_{\beta_g\circ\varphi}=\gamma_g\circ\pi_\varphi.\] 
In other words, the $\ast$-homomorphism 
$\pi_\varphi\colon A\to M(C)\cap \{h\}'$ is equivariant. Since the
action on $A$ is assumed to be continuous, the range of this map is 
contained in the $\gamma$-\emph{continuous part} of $M(C)\cap \{h\}'$, namely
\[M(C)_\gamma\cap \{h\}'=\{x\in M(C)\colon xh=hx \mbox{ and } 
 g\mapsto \gamma_g(x) \mbox{ is norm-continuous}\}.
\]
\end{rmk}

Recall a linear map $\varphi\colon A\to B$ between \ca{s} is \emph{positive} if $\varphi(A_+)\subseteq B_+$;
it is \emph{$n$-positive} if the amplification $\varphi^{(n)}=\varphi\otimes\mathrm{id}_{M_n}\colon M_n(A)\to M_n(B)$ is positive;
and it is \emph{completely positive} if it is $n$-positive for every $n\in\mathbb{N}$.

Note that a weighted $\ast$-homomorphism with weight $h$ is self-adjoint (positive) if and only if $h$ is self-adjoint (positive).
It is easy to see that a positively weighted $\ast$-homomorphism is even completely positive.
Hence, we obtain the following characterization of completely positive, order zero maps:

\begin{cor}
\label{prp:CharOrderZero}
A map between \ca{s} is completely positive, order zero if and only if it is positive and preserves zero-products.
In particular, every positive, zero-product preserving map between \ca{s} is automatically completely positive.
\end{cor}

\begin{rmk}
Following Sato, \cite{Sat21AlmostOrderZero}, we say that a positive (but not necessarily completely positive) map between \ca{s} is \emph{order zero} if it preserves zero-products of positive elements.
By \cite[Corollary~3.7]{Sat21AlmostOrderZero}, every $2$-positive, order zero map is completely positive.
On the other hand, not every positive, order zero map is automatically completely positive, and
the transpose map on $M_2(\mathbb{C})$ is a counterexample.

An alternative proof of \cite[Corollary~3.7]{Sat21AlmostOrderZero} can be obtained 
from \autoref{prp:CharOrderZero}
by noting that every $2$-positive, order zero map preserves zero-products.
Indeed, we first note that elements $a,b$ in a \ca{} satisfy $ab=0$ if and only if $(a^*a)(bb^*)=0$.
Now, let $\varphi\colon A\to B$ be a $2$-positive, order zero map. 
Then $\varphi$ satisfies the Kadison inequality $\varphi(a)^*\varphi(a)\leq\|\varphi\|\varphi(a^*a)$ for all $a\in A$.
Hence, if $a,b\in A$ satisfy $ab=0$, then $(a^*a)(bb^*)=0$ and therefore
\[
\varphi(a^*a)\varphi(bb^*)=0, \quad
\varphi(a)^*\varphi(a)\leq\|\varphi\|\varphi(a^*a), \andSep
\varphi(b)\varphi(b)^*\leq\|\varphi\|\varphi(bb^*).
\]
This implies $\varphi(a)^*\varphi(a)\varphi(b)\varphi(b)^*=0$, and thus $\varphi(a)\varphi(b)=0$, as desired.
\end{rmk}

\begin{prp}
\label{prp:CB}
Let $\varphi\colon A\to B$ be a bounded, linear map between \ca{s}.
If $\varphi$ preserves range-orthogonality, domain-orthogonality and zero-products, then $\varphi$ is completely bounded with $\|\varphi\|_{cb}=\|\varphi\|=\|\varphi^{**}(1)\|$.
\end{prp}
\begin{proof}
Set $h:=\varphi^{**}(1)\in B^{**}$.
By \autoref{prp:CharWeighted}, there exists a $\ast$-homo\-morphism $\pi\colon A\to B^{**}$ such that $\varphi(a)=h\pi(a)$ for all $a\in A$.
We have
\[
\|\varphi^{**}(1)\|
\leq \|\varphi^{**}\|
= \|\varphi\|
\leq \|\varphi\|_{cb}.
\]
Given $n\geq 1$, set $h^{(n)}=1_{M_n}\otimes h \in M_n(B^{**})$, which is the diagonal matrix with  diagonal entries all equal to $h$.
It follows that the amplification $\varphi^{(n)}$ satisfies $\varphi^{(n)}(x)=h^{(n)}\pi^{(n)}(x)$ for all $x\in M_n(A)$.
Using that $\pi^{(n)}$ is a $\ast$-homomorphism, we deduce that
\[
\|\varphi^{(n)}(x)\|
= \|h^{(n)}\pi^{(n)}(x)\|
\leq \|h^{(n)}\| \|\pi^{(n)}(x)\|
\leq \|h\| \|x\|
\]
and thus $\|\varphi^{(n)}\|\leq\|h\|$.
Since this holds for every $n$, we obtain $\|\varphi\|_{cb}\leq\|h\|$, as desired.
\end{proof}

\begin{rmk}
In \cite{GarThi22pre:AutomaticContinuity}, we show that \autoref{prp:CB} also holds for bounded linear maps that only preserve range-orthogonality or domain-orthogonality.
In particular, a bounded, range-orthogonality preserving map is automatically completely bounded.
We also show that a range-orthogonality preserving map from a unital \ca{} that has no one-dimensional irreducible representations is automatically bounded, and hence completely bounded.
\end{rmk}


\providecommand{\etalchar}[1]{$^{#1}$}
\providecommand{\href}[2]{#2}

\end{document}